\documentclass{article}
\usepackage{graphicx}

 \usepackage[english]{babel}   
  \usepackage{times}			
 
\usepackage[T1]{fontenc}
\usepackage[utf8]{inputenc} 
\usepackage[sort,nocompress]{cite}
\DeclareUnicodeCharacter{00A0}{~}

\usepackage{amssymb,amsmath,amscd,amsfonts,amsthm,bbm,mathrsfs,yhmath}
\usepackage[shortlabels]{enumitem}
\usepackage{hyperref}

\DeclareMathOperator{\dom}{dom}

\DeclareMathOperator{\inter}{int}

\DeclareMathOperator*{\argmin}{arg\,min}

\newcommand{\eqdef}{:=}

\newcommand{\bP}{{{\mathbb P}}} 
\newcommand{\bE}{{{\mathbb E}}} 
 
\newcommand{\bN}{{{\mathbb N}}} 

\newcommand{\sA}{{\mathsf A}}
\newcommand{\sB}{{\mathsf B}}
\newcommand{\sC}{{\mathsf C}}

\newcommand{\sX}{{\mathsf X}}

\newcommand{\maxmon}{{\mathscr M}} 
\newcommand{\Selec}{{\mathfrak S}} 

\newcommand{\mcB}{{\mathscr B}}

\newcommand{\mcF}{{\mathscr F}} 
\newcommand{\mcG}{{\mathscr G}} 
 
\newcommand{\mcL}{{\mathscr L}}

\newcommand{\cD}{{{\mathcal D}}}

\newcommand{\cA}{{{\mathcal A}}}

\newcommand{\bR}{{{\mathbb R}}}

\newcommand{\ps}[1]{\langle #1 \rangle}

\usepackage[textwidth=2cm, textsize=footnotesize]{todonotes}  
\theoremstyle{definition}
\newtheorem{theorem}{Theorem}
\newtheorem{lemma}[theorem]{Lemma}
\newtheorem{corollary}[theorem]{Corollary}
\newtheorem{proposition}[theorem]{Proposition}
\newtheorem{definition}{Definition}

\title{A Strong Law of Large Numbers \\ for Random Monotone Operators}

\author{Adil Salim\thanks{Microsoft Research, Redmond, WA, USA. \newline \texttt{adilsalim@microsoft.com}}}

\begin{document}
%
\maketitle
%


\begin{abstract} 
{\color{black} Random monotone operators are stochastic versions of maximal monotone operators which play an important role in stochastic nonsmooth optimization. Several stochastic nonsmooth optimization algorithms have been shown to converge to a zero of a mean operator defined as the expectation, in the sense of the Aumann integral, of a random monotone operator.

In this note, we prove a strong law of large numbers for random monotone operators where the limit is the mean operator. We apply this result to the empirical risk minimization problem appearing in machine learning. We show that if the empirical risk minimizers converge as the number of data points goes to infinity, then they converge to an expected risk minimizer.}
\end{abstract}

\section{Introduction}

Maximal monotone operators are set valued mappings which play an important role in various fields of convex analysis~\cite{bre-livre73,bau-com-livre11}, ranging from convex optimization to the analysis of Partial Differential Equations. {\color{black} Some recent developments in large scale convex optimization rely on stochastic versions of these maximal monotone operators.}

{\color{black} Indeed,} the set of maximal monotone operators $\maxmon(H)$ over a separable Hilbert space $H$ can be endowed with a topology~\cite[Chap. I]{att-79} (see also~\cite[Chap. III]{attouch1984variational}) such that $\maxmon(H)$ is a Polish space~\cite[Prop 1.1]{att-79}. Therefore one can study probability theory over the set $\maxmon(H)$~\cite[Chap. II]{att-79}. In particular, a random monotone operator is defined as a random variable with values in $\maxmon(H)$~\cite{bia-hac-16}. Random monotone operators were used to prove the convergence of the stochastic Forward Backward algorithm in~\cite{bia-hac-16, bia-hac-sal-(sub)jca17} where the expectation of a random monotone operator is defined through its Aumann integral~\cite{aum-65} (generalization of Lebesgue integral to set valued mappings). In this context, one may ask if random monotone operators admit a law of large numbers.

Various laws of large numbers for random sets have already been proven in the literature. Different class of random sets were considered (compact, unbounded...), see \textit{e.g.}~\cite{artstein1975strong,artstein1981law,Taylor1997,shapiro2007uniform,teran2008uniform,castaing2013law}. In particular, laws of large numbers for compact valued subdifferentials of random non convex functions were obtained in~\cite{shapiro2007uniform,teran2008uniform,castaing2013law}. The subdifferential of a convex function is a monotone operator, but the converse is false. In particular, the laws of large numbers mentioned above do not cover the case of general random monotone operators {\color{black} which are not subdifferentials}.

In this note, we prove a law of large numbers for random monotone operators, and apply it to the convergence of the empirical mean minimizers~\cite{bottou2008tradeoffs}. 

The next section provides some background knowledge on (random) monotone operators. Then, the main theorem is stated in section~\ref{sec:res}. Section~\ref{sec:proof} is devoted to the proof of the main result. The application to empirical risk minimization is provided in section~\ref{sec:erm}. Finally, we conclude in section~\ref{sec:ccl}. 
\subsection{Data availability statement}
Data sharing not applicable to this article as no datasets were generated or analysed during the current study.

\section{Background}

\label{sec:background}

In this section, we define maximal monotone operators, random monotone operators, and their expectation.

\subsection{Maximal monotone operators}
We review some basic material regarding maximal monotone operators. The proofs of these facts can be found in~\cite{bau-com-livre11}.

Let $H$ be a separable Hilbert space and let $I$ be the identity map over $H$. An operator $\sA$ over $H$ is a set valued mapping over $H$, \textit{i.e} a function from $H$ to the set of all subsets of $H$. An operator can be identified to its graph $G(\sA) \eqdef \{(x,y) \in H \times H, y \in \sA(x)\}$. The domain of $\sA$ is defined as $\dom(\sA) \eqdef \{x \in H, \sA(x) \neq \emptyset\}$. The inverse operator $\sA^{-1}$ is defined by $G(\sA^{-1}) = \{(y,x) \in H \times H, y \in \sA(x)\}$, the resolvent operator is defined by $J_{\sA} \eqdef (I+\sA)^{-1}$ and the set of zeros of $\sA$ is $Z(\sA) \eqdef \sA^{-1}(0)$. Note that $\ell \in Z(\sA)$ if and only if $\ell \in J_\sA(\ell)$. The operator $\sA$ is said monotone if the following condition holds: 
\begin{equation*}
\forall (x,y),(x',y') \in G(\sA), \ps{x-x',y-y'} \geq 0,
\end{equation*}
where $\ps{\cdot,\cdot}$ denotes the inner product of $H$. In this case, $J_{\sA}(x)$ is either the empty set or a singleton, \textit{i.e}, $J_{\sA}$ can be identified with a classical function $\dom(J_{\sA}) \to H$.

The monotone operator $\sA$ is said maximal, which we denote $\sA \in \maxmon(H)$, if $\dom(J_{\sA}) = H$. In this case, $J_{\sA} : H \to H$ is a $1$-Lipschitz continuous function. Minty's theorem states that the maximality of $\sA$ is equivalent to the maximality (for the inclusion ordering) of $G(\sA)$ in the set of all graphs of monotone operators over $H$~\cite{minty1962monotone}. 
Given $\gamma > 0$, the Yosida approximation of $\sA$ is the function defined by $\sA_\gamma(x) \eqdef \frac{x - J_{\gamma\sA}(x)}{\gamma}$. The function $\sA_\gamma$ is $1/\gamma$-Lipschitz continuous. 


Given two maximal monotone operators $\sA$ and $\sB$, the sum $\sA + \sB$ is defined by $(\sA + \sB)(x) \eqdef \sA(x) + \sB(x)$ where $\sA(x) + \sB(x)$ is the classical Minkowski sum of two sets. One can check that $\sA + \sB$ is a monotone operator, however, $\sA + \sB$ is not necessarily maximal~\cite[Page 54]{phelps2009convex}. {\color{black} Sufficient conditions for the maximality of $\sA + \sB$ include (i) the case where $\dom(\sB) = H$ (ii) the case where $\dom(\sA) \cap \inter(\dom(B)) \neq \emptyset$, where $\inter$ denotes the interior of a set.}

Consider the set $\Gamma_0(H)$ of convex lower semi-continuous and proper functions $F : H \to (-\infty,+\infty]$. Then, the subdifferential $\partial F$ of $F$ is a maximal monotone operator. In other words, $\maxmon_s(H) \eqdef \{\partial F, F \in \Gamma_0(H)\}$ is a subset of $\maxmon(H)$. {\color{black} Besides, $J_{\partial F}$ is the proximity operator of $F$ and $(\partial F)^{-1} = \partial F^{\ast}$ where $F^\ast$ is the Legendre-Fenchel transform of $F$. We denote by $\dom(F)$ the domain of $F$, \textit{i.e.}, $\dom(F) \eqdef \{x \in H, F(x) < \infty\}$. 

Finally,} let $C$ be a convex set and consider $F = \iota_C$ the convex indicator function of $C$, defined by $F(x) = 0$ if $x \in C$ and $F(x) = +\infty$ else. Then $F \in  \Gamma_0(H)$ and $\partial F$ is the normal cone $N_C$ to $C$. 

\subsection{Random monotone operators}
For every $x \in H$, consider the map $p_{x}$ from $\maxmon(H)$ to $H$ defined by $p_{x}(\sA) \eqdef J_{\sA}(x)$. The topology of R-convergence is the initial topology on $\maxmon(H)$ with respect to the family of functions $\{p_{x}, x \in H\}$. In other words, the R-topology is the coarsest topology on $\maxmon(H)$ that makes the functions $p_{x}$ continuous. Endowed with this topology, $\maxmon(H)$ is a Polish space~\cite[Lemme 2.1]{att-79} (metrizable, separable and complete). 

In the sequel, we consider a probability space $(\Xi, \mcG, \mu)$ such that $\mcG$ is $\sigma$-finite and $\mu$-complete, and a measurable map $A :(\Xi, \mcG, \mu) \to (\maxmon(H),\mcB(\maxmon(H)))$ (where $\mcB(\sX)$ denotes the Borelian sigma field over any topological space $\sX$). Such a measurable map is called a \textit{random monotone operator}. 

A normal convex integrand is a measurable map $f : (\Xi \times H, \mcG \otimes \mcB(H)) \to ((-\infty,+\infty] , \mcB((-\infty,+\infty]))$ such that for every $s \in \Xi$, $f(s,\cdot) \in \Gamma_0(H)$. Using~\cite[Theorem 2.3]{att-79}, $s \mapsto \partial f(s,\cdot)$ is a random monotone operator. 

\subsubsection{{\color{black}Mean operator}}
\label{sec:mean}
Let $\mcL^1(\Xi,\mcG,\mu)$ be the space of $\mcG$-measurable and $\mu$-integrable $H$-valued functions defined on $\Xi.$ For every $x \in H$, we define
\[
\Selec_x \eqdef 
\{ \varphi \in \mcL^1(\Xi,\mcG,\mu) \, : \, 
\varphi(s) \in A(s)(x) \ \text{for }\mu-\text{almost every (a.e.) } s \in \Xi \} \, .
\]
We shall prefer the notation $A(s,x)$ for the set $A(s)(x)$. Note that the set $\Selec_x$ might be empty. The mean operator $\cA$ of $A$ is defined by its Aumann integral~\cite{aum-65},
\[
\forall x \in H, \cA(x) \eqdef \left\{ \int \varphi d\mu \, : \, 
\varphi \in \Selec_x \right\}.
\]
We shall refer to $\cA$ as the expectation of $A$. 

One can check that $\cA$ is a monotone operator. Indeed, let $\varphi \in \Selec_x$ and $\varphi' \in \Selec_{x'}$. Then, $\mu$-a.e.,
\begin{equation*}
\ps{x-x',\varphi-\varphi'} \geq 0.
\end{equation*}
Therefore,
\begin{equation*}
\ps{x-x',\int \varphi d\mu - \int \varphi' d\mu} \geq 0,
\end{equation*}
which proves the monotonicity of $\cA$. However, the maximality of $\cA$ does not follow from the definition of $\cA$. 
\begin{definition}
\label{def}
The random monotone operator $A$ is said \textbf{integrable} if $\cA$ is a \textit{maximal} monotone operator.
\end{definition}
{\color{black}Various conditions can ensure the integrability of $A$. We can classify them depending on the nature of the measure $\mu$, discrete or not.

\begin{itemize}
\item If $\mu$ is a discrete measure, then $\cA$ can be written as a finite sum of maximal monotone operators: $\cA = \sum_{i = 1}^m B_i$. Therefore conditions for the maximality of $\cA$ can be obtained from the conditions for the maximality of a sum of maximal monotone operators. In particular, if $\cap_{i=1}^m \inter(\dom(\sB_i)) \neq \emptyset$, then $\cA$ is maximal using~\cite[Corollary 24.4]{bau-com-livre11}. 

Moreover, in the case where $A$ is a random subdifferential, $\cA$ can be written as a finite sum of subdifferentials: $\cA = \sum_{i = 1}^m \partial G_i$, where $G_i \in \Gamma_0(H)$. If the following interchange property holds: $\cA = \partial \sum_{i = 1}^m G_i$, then $\cA$ is maximal because $\cA$ is the subdifferential of $\sum_{i = 1}^m G_i \in \Gamma_0(H)$. The interchange property means that one can exchange the sum and the subdifferentiation $\partial$. General conditions under which the interchange property holds can be found in~\cite[Corollary 16.39]{bau-com-livre11}. In particular, if $\cap_{i=1}^m \inter(\dom(G_i)) \neq \emptyset$, then the interchange property holds using~\cite[Corollary 16.39]{bau-com-livre11} and $\cA$ is maximal. Finally, the interior of $\dom(G_i)$ can be replaced by the relative interior of $\dom(G_i)$ if $H$ is finite dimensional, see~\cite[Corollary 16.39]{bau-com-livre11}.

\item In the general case where $\mu$ is not necessarily discrete, general conditions ensuring the maximality of $\cA$ can be found in~\cite[Proposition 3.1]{bia-hac-16}. Another condition is domination: there exists a non-negative valued function $g \in L^1(\Xi,\mcG,\mu)$ such that for every $x \in H$, $A(s,x) \neq \emptyset$ and $\sup_{y \in A(s,x)} \|y\| \leq g(s)$ a.e., see~\cite[Example 2]{bia-16}. 

Moreover, in the subdifferential case where $A(s) = \partial f(s,\cdot)$ and $f$ is a normal convex integrand, $\cA$ can be written as the expectation w.r.t. $\mu$ of the subdifferentials $\partial f(s,\cdot)$. If the following interchange property holds: $\cA = \partial F$ where $F(x) = \bE_\xi(f(\xi,x))$, and if $F \in \Gamma_0(H)$, then $\cA$ is maximal, because $\cA$ is the subdifferential of $F \in \Gamma_0(H)$. The interchange property means that one can exchange the expectation $\bE$ and the subdifferentiation $\partial$. General conditions under which the interchange property holds can be found in~\cite{roc-wet-82}. In particular, if $\int |f(s,x)| d\mu(s) < \infty$ for every $x \in H$, we have both the interchange property and $F \in \Gamma_0(H)$, therefore $\cA$ is maximal. 
\end{itemize}}

\section{Main result}
\label{sec:res}

In this section we provide the main theorem and discuss our assumptions.

\begin{theorem}[Law of large numbers for random monotone operators]
\label{th}
Consider a family of i.i.d random variables $(\xi_n)_n$ from some probability space $(\Omega,\mcF,\bP)$ to $(\Xi,\mcG)$ with distribution $\mu$.
Assume that the random monotone operator $A$ is integrable and that for every $n \in \bN$, 
\begin{equation}
\label{eq:empmean}
\overline{A_n} \eqdef \frac{1}{n}\sum_{k = 1}^n A(\xi_k)
\end{equation}
is $\bP$-almost surely (a.s.) maximal. 

Then, $\overline{A_n} : (\Omega,\mcF,\bP) \to (\maxmon(H),\mcB(\maxmon(H)))$ is a random monotone operator and $\bP$-a.s, 
\begin{equation}
\label{eq:convergence}
    \overline{A_n} \underset{n \to \infty}{\longrightarrow} \cA,
\end{equation}
in the sense of R-convergence.
{\color{black}Moreover, if $A(s) = \partial f(s,\cdot)$ where $f$ is a normal convex integrand, then $\cA = \partial F$ where $F(x) = \bE_\xi(f(\xi,x))$.}

\end{theorem}
This theorem is a law of large numbers for the family of i.i.d random monotone operators $(A(\xi_n))_n$, where the limit is the expectation of $A$. 

Moreover, {\color{black} in the subdifferential case where $A(s) = \partial f(s,\cdot)$, Theorem~\ref{th} recovers a law of large numbers for subdifferentials of convex functions as a special case.}

{\color{black}
Let us now discuss the assumptions. Our first assumption is the integrability of $A$ (\textit{i.e} the maximality of $\cA$). Note that $A$ must be integrable for the convergence~\eqref{eq:convergence} to hold, since~\eqref{eq:convergence} is a convergence in the space of maximal monotone operators. Conditions under which $A$ is integrable are provided in Section~\ref{sec:mean}.
}

{\color{black}

Our second assumption is the a.s. maximality of $\overline{A_n}$. Note that $\overline{A_n}$ must be a.s. maximal for the convergence~\eqref{eq:convergence} to hold, since~\eqref{eq:convergence} is a convergence in the space of maximal monotone operators. We provide two sufficient conditions for the a.s. maximality of $\overline{A_n}$ in Proposition~\ref{prop1} and~\ref{prop2}.

\begin{proposition}
\label{prop1}
Denote $\cD$ the essential intersection of $\inter(\dom(A(s)))$, \textit{i.e.}, the set defined by $x \in \cD \Longleftrightarrow x \in \inter(\dom(A(s)))$ for $\mu$-a.e. $s$. If $\cD \neq \emptyset$, then for every $n$, $\overline{A_n}$ is a.s. maximal. 
\end{proposition}
\begin{proof}
Let $n \geq 1$ and $x \in \cD$. For every $k \in \{1, \ldots, n\}$, $\bP \left( x \in \inter(\dom(A(\xi_k))) \right) = 1$. Therefore, using independence, $\bP \left( x \in \cap_{k=1}^{n} \inter(\dom(A(\xi_k))) \right) = 1$. In particular, $\cap_{k=1}^{n} \inter(\dom(A(\xi_k))) \neq \emptyset$ a.s. Therefore, $\overline{A_n}$ is maximal a.s using~\cite[Corollary 24.4]{bau-com-livre11}. 
\end{proof}
\begin{proposition}
\label{prop2}
Assume that $A(s) = \partial f(s,\cdot)$ where $f$ is a normal convex integrand. Denote $\cD_s$ the essential intersection of $\inter(\dom(f(s,\cdot)))$, \textit{i.e.}, the set defined by $x \in \cD \Longleftrightarrow x \in \inter(\dom(f(s,\cdot)))$ for $\mu$-a.e. $s$. If $\cD_s \neq \emptyset$, then for every $n$, $\overline{A_n} = \partial \overline{f_n}$ where $\overline{f_n}(x) = \frac{1}{n}\sum_{k=1}^n f(\xi_k,x)$. In particular, $\overline{A_n}$ is a.s. maximal.

Finally, the interior of $\dom(f(s,\cdot))$ can be replaced by the relative interior of $\dom(f(s,\cdot))$ in the definition of $\cD_s$ if $H$ is finite dimensional.
\end{proposition}
\begin{proof}
Using $\cD_s \neq \emptyset$, we first obtain $\cap_{k=1}^{n} \inter(\dom(f(\xi_k,\cdot))) \neq \emptyset$ a.s. as in the previous proof. This ensures that the interchange property holds, \textit{i.e.}, $\overline{A_n} = \partial \overline{f_n}$. Since $\overline{f_n} \in \Gamma_0(H)$, $\overline{A_n}$ is maximal. Finally, if $H$ is finite dimensional, one can replace the the interior of $\dom(f(s,\cdot))$ by its relative interior in the definition of $\cD_s$ and this proposition is still valid with the same proof, see~\cite[Corollary 16.39]{bau-com-livre11}.
\end{proof}
}

{\color{black} Finally, we comment on the relationship between our assumptions.} The integrability of $A$ is not a consequence of the other assumptions. More generally, there is no logical relationship between the maximality of $\cA$ and the maximality of $\overline{A_n}$. To illustrate this, we shall use an example of two maximal monotone operators $\sB$ and $\sC$ provided in~\cite[Page 54]{phelps2009convex}, such that $\dom(\sB + \sC) = \{0\}$ but $\sB + \sC \neq N_{\{0\}}$ (and hence $\sB + \sC$ is not maximal). If $A$ is uniformly distributed over $\{\sB, \sC, N_{\{0\}}\}$, then, $\cA$ is maximal but with positive probability $\overline{A_2}$ is not maximal. If $A$ is uniform over $\{\sB, \sC\}$, then, with positive probability, $\overline{A_2}$ is maximal although $\cA$ is not maximal.

\section{Proof of the main result}
\label{sec:proof}
Since $\overline{A_2}$ is maximal, it is a random monotone operator using~\cite[Theorem 2.4]{att-79}. An alternative proof of the measurability of $\overline{A_2}$ is as follows: for every $y \in H$, $x = J_{\overline{A_2}}(y)$ is the solution to the monotone inclusion $0 \in (I - y)(x) + \frac12 A(\xi_1(\omega),x) + \frac12 A(\xi_2(\omega),x)$ for which the three operator splitting algorithm of~\cite{davis2017three} can be applied. This algorithm provides a sequence of iterates $(x_n(\omega))$ converging to $x$. One can show by induction that $\omega \mapsto x_n(\omega)$ is measurable. Therefore $J_{\overline{A_2}}(y)$ is also a random variable for every $y \in H$, which proves the measurability of $\overline{A_2}$~\cite[Lemma 2.1]{att-79}. Then, by induction, $\overline{A_n}$ is a random monotone operator for every $n$.

\begin{lemma}
\label{lem:zero}
Under the assumptions of Theorem~\ref{th}, if $x_\star \in Z(\cA)$ then, $$J_{\overline{A_n}}(x_\star) \longrightarrow x_\star,$$ as $n \to +\infty$, $\bP$-a.s.
\end{lemma}
\begin{proof}
Since $0 \in \cA(x_\star)$, there exists a measurable map $\varphi : (\Xi,\mcG,\mu) \to (H,\mcB(H))$ such that $\varphi$ is $\mu$-integrable, $\int \varphi d\mu = 0$ and $\varphi(s) \in A(s,x_\star)$ $\mu$-a.s.
Consider the random variables $\overline{\phi_n} = \frac{1}{n}\sum_{k=1}^n \varphi(\xi_k)$. Note that $\overline{\phi_n}$ is integrable, $\overline{\phi_n} \in \overline{A_n}(x_\star)$ $\bP$-a.s. and $\bE(\overline{\phi_n}) = 0$.

{\color{black} Let $\gamma >0$ and $x \in H$, then,
\begin{align*}
\|J_{\gamma \overline{A_n}}(x) - x_\star\|^2 =& \|x - x_\star\|^2 -2\gamma\ps{{\overline{A_n}}_\gamma(x), x - x_\star} + \gamma^2\|{\overline{A_n}}_\gamma(x)\|^2\\
=& \|x - x_\star\|^2 -2\gamma\ps{{\overline{A_n}}_\gamma(x), J_{\gamma \overline{A_n}}(x) - x_\star} - \gamma^2\|{\overline{A_n}}_\gamma(x)\|^2\\
=& \|x - x_\star\|^2 -2\gamma\ps{{\overline{A_n}}_\gamma(x)- \overline{\phi_n}, J_{\gamma \overline{A_n}}(x) - x_\star} - \gamma^2\|{\overline{A_n}}_\gamma(x)\|^2\\
& -2\gamma \ps{\overline{\phi_n},x- x_\star} +2\gamma \ps{\overline{\phi_n}, {\overline{A_n}}_\gamma(x)}\\
\leq& \|x - x_\star\|^2 -2\gamma \ps{\overline{\phi_n},x- x_\star} +\gamma^2 \|\overline{\phi_n}\|^2,
\end{align*}
where the last inequality comes from Young's inequality and monotonicity of $\overline{A_n}$.
Taking $\gamma = 1$ and $x = x_\star$ we get
\begin{equation*}
    \|J_{\overline{A_n}}(x_\star) - x_\star\| \leq \|\overline{\phi_n}\|, \quad \bP \text{-a.s.}
\end{equation*}}
Using the Strong Law of Large Numbers in Hilbert spaces (\cite[Corollary 7.10]{ledoux2013probability}) for $\overline{\phi_n}$ we have $\bP$-a.s., 
$$\|\overline{\phi_n}\| \longrightarrow_{n \to +\infty} 0.$$ and hence $\bP$-a.s.
$$\|J_{\overline{A_n}}(x_\star) - x_\star\| \longrightarrow_{n \to +\infty} 0,
$$
which concludes the proof.
\end{proof}
\begin{lemma}
\label{lem:mes}
Consider $z \in H$. Then, $A-z : x \mapsto A(x) - z$ is a random monotone operator and
\begin{equation}
\label{eq:A-z}
    J_{A-z}(y) = J_{A}(y+z), \quad \forall y \in H.
\end{equation}
\end{lemma}
\begin{proof}
Equation~\eqref{eq:A-z} is well known and can be found for example in~\cite{bau-com-livre11}. We provide a full proof for the sake of completeness.
For any $y \in H$, the inclusion $y \in x + (A-z)(x)$ (where $x$ is the unknown) is equivalent to $y + z \in x + A(x)$ and hence admits a unique solution $x = J_{A}(y+z)$. This implies that $A-z$ is $\mu$-a.s a maximal monotone operator, and $J_{A-z}(y) = J_{A}(y+z)$. 

We also see that $s \mapsto J_{A(s)-z}(y)$ is measurable for every $y \in H$ and hence, $A-z$ is a random monotone operator (see \cite[Lemme 2.1]{att-79}).
\end{proof}
\subsection{{\color{black} End of the proof of Theorem~\ref{th}}}
We now prove Theorem~\ref{th}. 
Consider $x \in H$. Since $\dom(J_{\cA}) = H$, $x \in \dom(J_{\cA})$. Therefore, there exists a unique $(y,z) \in G(\cA)$ such that $x = y + z$. Therefore, $0 \in \cA(y) - z$,  \textit{i.e}, $y \in Z(\cA - z)$. Using Lemma~\ref{lem:mes} and the maximality of $\overline{A_n}$, $A-z$ is a random monotone operator and $\frac{1}{n}\sum_{k=1}^n (A(\xi_k) - z) = \overline{A_n} - z$ is $\mu$-a.s. maximal. Moreover, $A-z$ is $\mu$-integrable with $\int (A-z) d\mu  = \cA - z$. Applying Lemma~\ref{lem:zero} to the random monotone operator $A-z$, we have $\bP$-a.s, 
\begin{equation}
\label{eq:cv}
J_{\overline{A_n} - z}(y) \longrightarrow y.
\end{equation}
Using $y = J_{\cA - z}(y)$, $x = y+z$ and Lemma~\ref{lem:mes}, the convergence~\eqref{eq:cv} can be rewritten as follows: for every $x \in H$, there exists a probability one event $\Omega_x \subset \Omega$ such that for every $\omega \in \Omega_x$,
$$J_{\overline{A_n}(\omega)}(x) \longrightarrow J_{\cA}(x).$$
We now show that $\Omega_x$ can be taken independent of $x$. Consider a dense countable subset $D$ of $H$, and the probability one event $\tilde{\Omega} = \bigcap_{x \in D} \Omega_x$. For every $\omega \in \tilde{\Omega},$ we have for every $x \in D$,
$$J_{\overline{A_n}(\omega)}(x) \longrightarrow J_{\cA}(x).$$
Consider $x_0 \in H$. We shall prove that for every $\omega \in \tilde{\Omega},$ we also have 
$$J_{\overline{A_n}(\omega)}(x_0) \longrightarrow J_{\cA}(x_0).$$
Let $\varepsilon >0$ and $x \in D$ such that $\|x-x_0\| < \varepsilon /3$. There exists $n_0 \in \bN$ such that for every $n \geq n_0$, $\|J_{\overline{A_n}(\omega)}(x) - J_{\cA}(x)\| < \varepsilon /3$. Let us decompose 
\begin{align*}
&\|J_{\overline{A_n}(\omega)}(x_0) - J_{\cA}(x_0)\| \\
\leq& \|J_{\overline{A_n}(\omega)}(x) - J_{\cA}(x)\| + \|J_{\overline{A_n}(\omega)}(x_0) - J_{\overline{A_n}(\omega)}(x)\| + \|J_{\cA}(x_0) - J_{\cA}(x)\|.
\end{align*}
Since resolvents are $1$-Lipschitz continuous, $\|J_{\overline{A_n}(\omega)}(x_0) - J_{\cA}(x_0)\| < \varepsilon$ for every $n \geq n_0$. We proved that for every $\omega \in \tilde{\Omega}$, $J_{\overline{A_n}(\omega)}(x) \longrightarrow J_{\cA}(x),$ for every $x \in H$, \textit{i.e.}, $\overline{A_n}(\omega) \longrightarrow \cA$, {\color{black}by definition of the R-convergence.} 

{\color{black}
In the case where $A(s) = \partial f(s,\cdot)$, we can show that $\cA = \partial F$. To this end, we start by showing that for every $x \in H$, 
\begin{equation}
\label{eq:incl}
\cA(x) \subset \partial F(x).
\end{equation}
Consider $x \in H$. If $\cA(x) = \emptyset$, the statement is trivial. Else, let $g \in \cA(x)$. There exists $\varphi \in \Selec_x$ such that $\int \varphi d\mu = g$. In particular, $\varphi(s) \in A(s,x) = \partial f(s,x)$ for a.e. $s$. Using the definition of the subdifferential, for every $y \in H$,
\begin{equation}
f(s,x) + \ps{\varphi(s), y - x} \leq f(s,y),
\end{equation}
for a.e. $s$. Integrating w.r.t. $\mu$ and using $\int \varphi d\mu = g$,
\begin{equation}
F(x) + \ps{g, y - x} \leq F(y).
\end{equation}
Therefore, $g \in \partial F(x)$, which proves that $\cA(x) \subset \partial F(x)$. In particular, $G(\cA) \subset G(\partial F)$, where $G$ denotes the graph. Using the convexity of $f(s,\cdot)$, one can prove that $F$ is convex. Therefore $\partial F$ is monotone, but not necessarily maximal \textit{a priori}. However, $G(\cA) \subset G(\partial F)$, and $\cA$ is maximal by assumption. Therefore, $G(\cA) = G(\partial F)$ which is equivalent to $\cA = \partial F$.}

{\color{black}
\section{Application to empirical risk minimization}
\label{sec:erm}

We now provide a consequence of the law of large numbers for random monotone operators. More precisely, we characterize $Z(\overline{A_n})$ as a subset of $Z(\cA)$ as $n \to \infty$. 

A random variable $\ell$ is an a.s. cluster point of the sequence $(x_n)$ of random variables if there exists a probability one event $\tilde{\Omega}$ such that for every $\omega \in \tilde{\Omega}$ there exists a subsequence of $x_n(\omega)$ converging to $\ell(\omega)$. The subsequence of $x_n(\omega)$ is called a random subsequence of $x_n$.

\begin{corollary}
\label{coro}
Let $(x_n)$ be a sequence of $H$-valued random variables such that $x_n \in Z(\overline{A_n})$ a.s. {\color{black} Assume that $\overline{A_n}$ is a.s. maximal and that $A$ is integrable with expectation $\cA$.}
Then, every a.s. cluster point $\ell$ of $(x_n)$ is a.s. a zero of $\cA$.
\end{corollary}

\begin{proof}
Consider a random subsequence of $(x_n)$ converging a.s. to $\ell$. This random subsequence is still denoted $(x_n)$. 
Denote $\overline{A_{n,\gamma}}$ the Yosida approximation of $\overline{A_n}$ and set $\gamma=1$. For every $n \geq 0$, $\overline{A_{n,\gamma}}(x_n) = 0$. Therefore,
\begin{align*}
    \|\ell - J_{\cA}(\ell)\| = \|\cA_\gamma(\ell)\| &\leq \|\cA_\gamma(\ell) - \overline{A_{n,\gamma}}(\ell)\| + \|\overline{A_{n,\gamma}}(\ell) - \overline{A_{n,\gamma}}(x_n)\|\\ &\leq \|J_{\cA}(\ell) - J_{\overline{A_{n}}}(\ell)\| + \|\ell - x_n\|.
\end{align*}
Since $\overline{A_n}$ and $\cA$ are maximal, we can use the law of large numbers (Theorem~\ref{th}): letting $n \to +\infty$ we obtain $\ell = J_{\cA}(\ell)$ a.s.
\end{proof}

The existence of cluster points is usually established independently using compactness arguments~\cite{dal2012introduction}.

\subsection{Unregularized empirical risk minimization}

Many machine learning and signal processing problems require to solve the so-called \textit{expected risk minimization} problem
\begin{equation}
\label{eq:trm}
\min_{x \in H} F(x) := \bE_\xi(f(\xi,x))
\end{equation}
where $H = \bR^d$, $\xi$ a random variable, $f$ is a normal convex integrand such that $f(\xi,x)$ is integrable. 

In these contexts, $\xi$ represents some random data with unknown distribution and hence evaluating $F$ is prohibitive. In practice, a number $n$ of i.i.d realizations $(\xi_k)$ of the data $\xi$ is given and the expected risk minimization is approximated by the \textit{empirical risk minimization} problem 
\begin{equation}
    \label{eq:erm}
\min_{x \in H} \overline{f_n}(x) := \frac{1}{n}\sum_{k = 1}^n f(\xi_k,x),
\end{equation} 
where $\xi_k$ are i.i.d copies of $\xi$. The empirical risk minimization is usually performed using some optimization algorithm. The output of the optimization algorithm is typically a minimizer $x_n$ of $\overline{f_n}$. A first consequence of Corollary~\ref{coro} is a characterization of the a.s. cluster points of $(x_n)$ as minimizers of $F$.

\begin{corollary}
\label{coro2}
Let $(x_n)$ be a sequence of $H$-valued random variables such that $x_n \in \argmin \overline{f_n}$ a.s. {\color{black} Assume that $f(\xi,x)$ is integrable for every $x \in H$.}
Then, every a.s. cluster point $\ell$ of $(x_n)$ is a.s. a minimizer of $F$.
\end{corollary}
\begin{proof}
We apply Corollary~\ref{coro} to the random monotone operator $A(s) \eqdef \partial f(s,\cdot)$. Recall that $\overline{A_n}(x) = \frac{1}{n}\sum_{k = 1}^n \partial f(\xi_k,x)$. We make several uses of~\cite[Page 179]{roc-wet-82} which implies that if $\bE |f(\xi,x)| < \infty$ for every $x \in H$, then the interchange property holds. 
\begin{itemize}
\item First, we show that $A$ is integrable. Using~\cite[Page 179]{roc-wet-82} and $\bE |f(\xi,x)| < \infty$ for every $x \in H$, the interchange property holds: 
\begin{equation}
\label{eq:interchange-exp}
\cA(x) = \partial F(x).
\end{equation}
Since $F \in \Gamma_0(H)$, $\cA$ is maximal, therefore, $A$ is integrable.
\item Then, we show that $\overline{A_n}$ is maximal. We view averaging as taking expectation w.r.t. an empirical distribution. Using~\cite[Page 179]{roc-wet-82} and $\frac{1}{n} \sum_{k=1}^n |f(\xi_k,x)| < \infty$ a.s. (which follows from $\bE |f(\xi,x)| < \infty$) for every $x \in H$, the interchange property holds a.s.: 
\begin{equation}
\label{eq:interchange-emp}
\overline{A_n}(x) = \partial \overline{f_n}(x).
\end{equation}
Since $\overline{f_n} \in \Gamma_0(H)$, $\overline{A_n}$ is a.s. maximal.
\item Finally, we show that $x_n \in Z(\overline{A_n})$ a.s. We know that $x_n \in \argmin \overline{f_n}$, which implies that $x_n \in Z(\partial \overline{f_n})$. Using~\eqref{eq:interchange-emp}, $x_n \in Z(\overline{A_n})$ a.s.   
\end{itemize}
Using Corollary~\ref{coro}, every a.s. cluster point of $(x_n)$ is a.s. a zero of $\cA$. Using~\eqref{eq:interchange-exp}, a zero of $\cA$ is a minimizer of $F$, which concludes the proof.
\end{proof}

The last corollary characterizes the cluster points of the minimizers of the empirical problem~\eqref{eq:erm} as minimizers of the expected problem~\eqref{eq:trm}. This result seems natural since $\overline{f_n}(x)$ converges to $F(x)$ a.s. But the proof of this corollary relies on the law of large numbers for some random monotone operators which are subdifferentials.

Other methods to prove Corollary~\ref{coro2} include Epi-convergence techniques~\cite{attouch1984variational,aze1988convergence,dal2012introduction}. Indeed, showing Epi-convergence of $\overline{f_n}$ to $F$, \textit{i.e.}, showing that for every $x \in H$,
\begin{align}
F(x) \leq& \liminf \overline{f_n}(x_n)& \text{for every sequence } (x_n) \text{ converging to } x \nonumber\\
F(x) \geq& \limsup \overline{f_n}(x_n)& \text{for at least one sequence } (x_n) \text{ converging to } x, \label{eq:epi}
\end{align}
would lead to the same conclusion as Corollary~\ref{coro2}.

Epi-convergence techniques are more general than ours since they cover the convergence of arbitrary sequences of functions satisfying~\eqref{eq:epi}. However, these techniques seem to be less suitable than ours for our specific case of sequences of empirical averages of convex functions. Indeed, we proved Corollary~\ref{coro2} assuming convexity and integrability of $f$ only.

Although the assumption that $f(\xi,x)$ is integrable holds in practice, this assumption is not necessary for Corollary~\ref{coro2} to hold. This assumption is only used to ensure that the interchange properties~\eqref{eq:interchange-exp} and~\eqref{eq:interchange-emp} hold.

\subsection{Empirical risk minimization with structured regularization}

Several machine learning and signal processing problems require to solve a \textit{regularized expected risk minimization} problem
\begin{equation}
\label{eq:trm-reg}
\min_{x \in H} \bE_\xi(f(\xi,x)) + R(Lx), \qquad F(x) := \bE_\xi(f(\xi,x)),
\end{equation}
where $H = \bR^d$, $\xi$ a random variable, $f$ is a normal convex integrand such that $f(\xi,x)$ is integrable, $K = \bR^p$, $R \in \Gamma_0(K)$ and $L$ is a $p \times d$ real matrix, \textit{i.e.}, a linear operator $H \to K$.
In these contexts, $\xi$ represents some random data with unknown distribution and hence evaluating $F$ is prohibitive. Moreover, $R(Lx)$ represents a structured regularization encoding constraints or sparsity for example. In practice, a number $n$ of i.i.d realizations $(\xi_k)$ of the data $\xi$ is given and the expected risk minimization is approximated by the \textit{regularized empirical risk minimization} problem 
\begin{equation}
    \label{eq:erm-reg}
\min_{x \in H} \frac{1}{n}\sum_{k = 1}^n f(\xi_k,x) + R(Lx), \qquad \overline{f_n}(x) := \frac{1}{n}\sum_{k = 1}^n f(\xi_k,x).
\end{equation} 

When $R \equiv 0$, the regularized empirical risk minimization boils down to \eqref{eq:erm} and is usually performed using some optimization algorithm. In this case, Corollary~\ref{coro2} can be applied to the convergence as $n \to \infty$ of the output of the optimization algorithm. 

When $R \not\equiv 0$ and $L = I$ the identity matrix, the regularized empirical risk minimization can also be performed using some optimization algorithm such as the proximal stochastic gradient algorithm~\cite{atc-for-mou-14,gorbunov2020unified}. The latter algorithm relies on the evaluation of the proximity operator of $R$, a.k.a. $J_{\partial R}$, which can be computed in closed form in many cases\footnote{\url{www.proximity-operator.net}}. Corollary~\ref{coro2} can easily be adapted to this case ($R \not\equiv 0$ and $L = I$) to study the convergence as $n \to \infty$ of the output of the optimization algorithm.

However, when $R \not\equiv 0$ and $L \neq I$, the proximity operator of $R(Lx)$ is usually hard to compute. In this case, primal--dual optimization methods~\cite{condat2019proximal} allow to solve~\eqref{eq:erm-reg} without computing the proximity operator of $R(Lx)$ explicitly: they rely on the proximity operator of $R$ and matrix vector multiplications involving $L$ only. They are therefore widely used for solving Problem~\eqref{eq:erm-reg}.

The output of a primal--dual optimization algorithm is a primal--dual optimal point $(x_n,y_n) \in H \times K$, \textit{i.e.}, a solution to:
\begin{equation}
\label{eq:kkt}
\left\{
    \begin{array}{ll}
        0 \in \partial \overline{f_n}(x_n) + L^T y_n \\
        0 \in - L x_n + \partial R^\ast(y_n)
    \end{array}
\right.,
\end{equation}
see~\cite{condat2019proximal}. In particular, $x_n$ is a minimizer of~\eqref{eq:erm-reg}, see~\cite{condat2019proximal}.

A second consequence of Corollary~\ref{coro} is a characterization of the a.s. cluster points of $(x_n,y_n)$ as primal--dual optimal points for Problem~\eqref{eq:trm-reg}:
\begin{equation}
\label{eq:kkt2}
\left\{
    \begin{array}{ll}
        0 \in \partial F(\ell) + L^T m \\
        0 \in - L \ell + \partial R^\ast(m).
    \end{array}
\right.
\end{equation}
A key element of the proof is that $(x_n,y_n)$ (resp. $(\ell,m)$) can be seen as a zero of a monotone operator $\overline{A_n}$ (resp. $\cA$). Besides, neither $\overline{A_n}$ nor $\cA$ are subdifferentials.

\begin{corollary}
\label{coro3}
Let $(x_n, y_n)$ be a sequence of $H \times K$-valued random variables such that $(x_n, y_n)$ is primal--dual optimal for~\eqref{eq:erm-reg}, \textit{i.e.}, satisfies~\eqref{eq:kkt} a.s. {\color{black} Assume that $f(\xi,x)$ is integrable for every $x \in H$.}
Then, every a.s. cluster point $(\ell,m)$ of $(x_n, y_n)$ is a.s. primal--dual optimal for~\eqref{eq:trm-reg}, \textit{i.e.}, satisfies~\eqref{eq:kkt2} a.s. In particular, $\ell$ is a minimizer of~\eqref{eq:trm-reg}.
\end{corollary}
\begin{proof}
We apply Corollary~\ref{coro} to the operator
\begin{equation*}
A(s)(x,y) \eqdef \begin{bmatrix} \partial f(s,x) + L^T y \\ - L x + \partial R^\ast(y) \end{bmatrix},
\end{equation*}
using vector notations.
\begin{itemize}
\item First, we show that $A$ is a random monotone operator over $H \times K$. $A$ can be decomposed as
\begin{equation}
\label{eq:decompo}
A(s)(x,y) = \begin{bmatrix} \partial f(s,x) \\ \partial R^\ast(y) \end{bmatrix} + \begin{bmatrix} L^T y \\ - L x \end{bmatrix},
\end{equation}
where the first term is the subdifferential of normal convex integrand $(s,(x,y)) \mapsto f(s,x) + R^\ast(y)$, hence a random monotone operator, and the second term is the deterministic  skew symmetric operator whose matrix is given by
\begin{equation*}
S \eqdef \begin{bmatrix} 0 & L^T \\ - L & 0\end{bmatrix}.
\end{equation*}
One can check that $\ps{Sz,z} = 0$ for any $z \in H \times K$ using skew symmetry. Therefore $S$ is monotone. Moreover $(I + S)$ is regular. Indeed, if $(I+S)z = 0$ then $\ps{z,(I+S)z} = \|z\|^2 = 0$ which implies $z = 0$. Therefore, the resolvent $(I+S)^{-1}$ is well defined everywhere, which implies the maximality of $S$. Finally, $S$ is maximal monotone and deterministic, hence a random monotone operator. 

Both terms in~\eqref{eq:decompo} are random monotone operators. Their sum is also a.s. maximal because $S$ has a full domain. Finally, $A$ is a random monotone operator using the same reasoning as in the beginning of Section~\ref{sec:proof}. Note that $A(s)$ is not a subdifferential, because a skew symmetric operator is not a subdifferential.
\item Then, we show that $A$ is integrable. As in the proof of Corollary~\ref{coro2}, the interchange property holds for $\partial F$. Therefore the expectation of $A$ can be written
\begin{equation}
\label{eq:interchange-exp-2}
\cA(x,y) = \begin{bmatrix} \partial F(x) + L^T y \\ - L x + \partial R^\ast(y) \end{bmatrix},
\end{equation}
which is maximal monotone using the same decomposition as~\eqref{eq:decompo}. Therefore, $A$ is integrable.
\item Next, we show that $\overline{A_n}$ is maximal. As in the proof of Corollary~\ref{coro2}, the interchange property holds for $\partial \overline{f_n}$, \textit{i.e.}, $\overline{\partial f_n} = \partial \overline{f_n}$. Therefore $\overline{A_n}$ can be written
\begin{equation}
\label{eq:interchange-emp-2}
\overline{A_n}(x,y) = \begin{bmatrix} \partial \overline{f_n}(x) + L^T y \\ - L x + \partial R^\ast(y) \end{bmatrix},
\end{equation}
which is a.s. maximal monotone using the same decomposition as~\eqref{eq:decompo}.
\item Finally, we show that $(x_n, y_n) \in Z(\overline{A_n})$ a.s. We know that $(x_n,y_n)$ satisfies~\eqref{eq:kkt}. Using~\eqref{eq:interchange-emp-2}, $(x_n,y_n) \in Z(\overline{A_n})$ a.s. 
\end{itemize}
Using Corollary~\ref{coro}, every a.s. cluster point of $(x_n)$ is a.s. a zero of $\cA$. Using~\eqref{eq:interchange-exp-2}, a zero of $\cA$ satisfies~\eqref{eq:kkt2}, which concludes the proof.
\end{proof}

The last corollary characterizes the cluster points of the primal--dual optimal points of the empirical problem~\eqref{eq:erm-reg} as primal--dual optimal points of the expected problem~\eqref{eq:trm-reg}. The proof of this corollary relies on the law of large numbers for random monotone operators which are not subdifferentials.

Since primal--dual optimal points are also saddle points of a Lagrangian function (see~\cite{condat2019proximal}), Corollary~\ref{coro3} could be obtained using Epi-convergence techniques~\cite{attouch1984variational,aze1988convergence,dal2012introduction}. But these techniques are more generic and therefore less suitable for our specific problem.

Finally, Corollary~\ref{coro3} can easily be extended to handle a random matrix $L(\xi)$, \textit{i.e.}, to the problem 
\begin{equation}
\label{eq:sto-lin-cons}
\min_{x \in H} \bE_\xi(f(\xi,x)) + R(\bE_\xi(L(\xi))x), 
\end{equation}
where $L(\xi)$ is a random matrix. Problem~\eqref{eq:sto-lin-cons} is quite general and covers stochastic linear constraints for example. Indeed, by taking $R = \iota_{\{b\}}$ where $b \in K$, Problem~\eqref{eq:sto-lin-cons} boils down to
\begin{equation}
\min_{x \in H} \bE_\xi(f(\xi,x)), \quad \text{s.t.} \quad \bE(L(\xi))x = b. 
\end{equation}

In conclusion, the law of large numbers for random monotone operators provides a versatile framework for studying the convergence of solutions of empirical problems appearing in machine learning and signal processing. 
}

\section{Conclusion}
\label{sec:ccl}

We proved a law of large numbers for random monotone operators. This work opens the door to the study of random monotone operators as random elements. An interesting question is whether their exists an universal distribution for random monotone operators, as the Gaussian distribution for real random variables, or other probabilistic objects, see \textit{e.g}~\cite{le2013uniqueness,ledoux2013probability,marvcenko1967distribution}.

\bibliographystyle{plain} 
 \newcommand{\noop}[1]{} \def\cprime{$'$} \def\cdprime{$''$} \def\cprime{$'$}

\end{document}